\documentclass[12pt, a4paper, leqno]{article}
\usepackage[margin=2.1cm]{geometry}
\usepackage[english]{babel}

\usepackage[runin]{abstract}
\usepackage{titling}

\usepackage{amsmath}
\usepackage{amsthm}
\usepackage{amsfonts}
\usepackage{amssymb}
\usepackage{eufrak}
\usepackage{stmaryrd}
\usepackage{bbm}
\usepackage{mathtools}
\usepackage{clrscode}
\usepackage{subscript}
\usepackage{enumitem}
\usepackage{mdwlist}
\usepackage{xspace}
\usepackage[backend=bibtex,style=numeric,firstinits=true,sorting=nyt,maxnames=4]{biblatex}

\abslabeldelim{.}

\newcommand{\keywordsname}{Key words}
\makeatletter
\newcommand{\keywords}[1]{%
\begin{@bstr@ctlist}
\hspace*{\abstitleskip}{\abstractnamefont\keywordsname\@bslabeldelim}\abstracttextfont\
#1%
\par\end{@bstr@ctlist}
}
\makeatother

\newcommand{\subjclassname}{Mathematics subject classification}
\makeatletter
\newcommand{\subjclass}[2][2010]{%
\begin{@bstr@ctlist}
\hspace*{\abstitleskip}{\abstractnamefont\subjclassname\ (#1)\@bslabeldelim}\abstracttextfont\
#2%
\par\end{@bstr@ctlist}
}
\makeatother

\makeatletter
\def\and{
	\end{tabular}%
	and%
	\begin{tabular}[t]{c}}%
\makeatother

\makeatletter
\def\thanks#1{
\protected@xdef\@thanks{\@thanks
\protect\footnotetext[\the\c@footnote]{#1}}%
}
\makeatother

\makeatletter
\let\addresses\@empty      
\newcommand{\address}[2][]{\g@addto@macro\addresses{\address{#1}{#2}}}
\newcommand{\curraddr}[2][]{\g@addto@macro\addresses{\curraddr{#1}{#2}}}
\newcommand{\email}[2][]{\g@addto@macro\addresses{\email{#1}{#2}}}
\newcommand{\urladdr}[2][]{\g@addto@macro\addresses{\urladdr{#1}{#2}}}
\def\enddoc@text{
  \ifx\@empty\addresses \else\@setaddresses\fi}
\AtEndDocument{\enddoc@text}
\def\emailaddrname{e-mail}
\def\@setaddresses{\par
  \nobreak \begingroup
  \interlinepenalty\@M
  \def\address##1##2{\begingroup%
    \par\addvspace\bigskipamount
    \@ifnotempty{##1}{(\ignorespaces##1\unskip) }%
    {\noindent\ignorespaces##2}\par\endgroup}%
  \def\email##1##2{\begingroup
    \@ifnotempty{##2}{\nobreak\noindent\emailaddrname
      \@ifnotempty{##1}{, \ignorespaces##1\unskip}\/:\space
      \ttfamily##2\par}\endgroup}%
  \addresses
  \endgroup
}

\makeatother

\makeatletter
\def\cstar#1{\expandafter\@cstar\csname c@#1\endcsname}
\def\@cstar#1{\ifcase#1\or $\ast$\or $\ast\ast$\or $\ast\ast\ast$\fi}
\AddEnumerateCounter{\cstar}{\@cstar}{$\ast\ast\ast$}
\makeatother

\newlist{conditions}{enumerate}{1}
\newlist{iconditions}{enumerate}{1}
\newlist{questions}{enumerate}{1}
\setlist[conditions]{label=\normalfont(\alph*),ref=\normalfont\alph*}
\setlist[iconditions]{label=\normalfont(\roman*),ref=\normalfont\roman*}
\setlist[questions]{label=\normalfont(Q\arabic*),ref=\normalfont Q\arabic*}

\newcommand{\rank}{\func{rank}}
\newcommand{\SC}{\mathcal{S}}
\newcommand{\Z}{\mathbb{Z}}
\newcommand{\R}{\mathbb{R}}
\newcommand{\C}{\mathcal{C}}
\newcommand{\F}{\mathbb{F}}
\newcommand{\HB}{\mathbb{H}}
\newcommand{\CB}{\mathbb{C}}
\newcommand{\SB}{\mathbb{S}}
\newcommand{\K}{\mathbb{K}}
\newcommand{\FC}{\mathcal{F}}
\newcommand{\VB}{\func{VB}}
\newcommand{\VBC}{\func{VB}_{\CB}}
\newcommand{\PB}{\mathbb{P}}
\newcommand{\cupproduct}{\mathbin{\smile}}
\newcommand{\Halg}{H_{\mathrm{alg}}}
\newcommand{\V}{\mathbb{V}}
\newcommand{\Spec}{\func{Spec}}
\newcommand{\HCalg}{H_{\CB\mhyphen\mathrm{alg}}}
\newcommand{\HCstr}{H_{\CB\mhyphen\mathrm{str}}}
\newcommand{\HCstreven}{H_{\CB\mhyphen\mathrm{str}}^{\mathrm{even}}}
\newcommand{\Heven}{H^{\mathrm{even}}}
\newcommand{\Hsph}{H_{\mathrm{sph}}}
\newcommand{\VBCstr}{\VB_{\CB\mhyphen\mathrm{str}}}
\newcommand{\ch}{\func{ch}}
\newcommand{\Q}{\mathbb{Q}}
\newcommand{\KF}{K_{\F}}
\newcommand{\KFstr}{K_{\F\mhyphen\mathrm{str}}}
\newcommand{\KFrk}{K_{\F}^{\mathrm{(rk)}}}
\newcommand{\KFcrk}{K_{\F}^{\mathrm{(crk)}}}
\newcommand{\GammaF}{\Gamma_{\F}}
\newcommand{\GammaC}{\Gamma_{\CB}}
\newcommand{\GammaR}{\Gamma_{\R}}
\newcommand{\KCstr}{K_{\CB\mhyphen\mathrm{str}}}
\newcommand{\rk}{{\normalfont{(rk)}}\xspace}

\newtheorem{theorem}{Theorem}[section]
\newtheorem{corollary}[theorem]{Corollary}
\newtheorem{proposition}[theorem]{Proposition}
\newtheorem{lemma}[theorem]{Lemma}
\theoremstyle{definition}
\newtheorem*{notation}{Notation}
\newtheorem{remark}[theorem]{Remark}
\theoremstyle{remark}
\newtheorem{case}{\indent Case}

\mathchardef\mhyphen="2D

\newtagform{c-lower-index}[c\textsubscript]{($}{$)}

\DeclareFieldFormat{title}{#1}
\DeclareFieldFormat
  [article,inbook,incollection,inproceedings,patent,thesis,unpublished]
  {title}{#1\isdot}
\DeclareFieldFormat{journaltitle}{#1}
\DeclareFieldFormat{booktitle}{#1}

\renewbibmacro*{in:}{%
  }

\DeclareFieldFormat[article,inbook,inproceedings]{pages}{#1}

\DeclareFieldFormat[article,periodical]{series}{
  \ifinteger{#1}
    {\mkbibparens{#1}}
    {\ifbibstring{#1}{\bibstring{#1}}{#1}}}

\DeclareBibliographyDriver{book}{%
  \usebibmacro{bibindex}%
  \usebibmacro{begentry}%
  \usebibmacro{author/editor+others/translator+others}%
  \setunit{\labelnamepunct}\newblock
  \usebibmacro{maintitle+title}%
  \newunit
  \printlist{language}%
  \newunit\newblock
  \usebibmacro{byauthor}%
  \newunit\newblock
  \usebibmacro{byeditor+others}%
  \newunit\newblock
  \printfield{edition}%
  \newunit\newblock
  \usebibmacro{series+number}%
  \iffieldundef{maintitle}
    {\printfield{volume}%
     \printfield{part}}
    {}%
  \newunit
  \printfield{volumes}%
  \newunit\newblock
  \printfield{note}%
  \newunit\newblock
  \usebibmacro{publisher+location+date}%
  \newunit\newblock
  \usebibmacro{chapter+pages}%
  \newunit
  \printfield{pagetotal}%
  \newunit\newblock
  \iftoggle{bbx:isbn}
    {\printfield{isbn}}
    {}%
  \newunit\newblock
  \usebibmacro{doi+eprint+url}%
  \newunit\newblock
  \usebibmacro{addendum+pubstate}%
  \setunit{\bibpagerefpunct}\newblock
  \usebibmacro{pageref}%
  \usebibmacro{finentry}}

\renewbibmacro*{publisher+location+date}{%
  \printlist{publisher}%
  \setunit*{\addcomma\space}%
  \printlist{location}%
  \setunit*{\addcomma\space}%
  \usebibmacro{date}%
  \newunit}

\DeclareBibliographyDriver{online}{%
  \usebibmacro{bibindex}%
  \usebibmacro{begentry}%
  \usebibmacro{author/editor+others/translator+others}%
  \setunit{\labelnamepunct}\newblock
  \usebibmacro{title}%
  \newunit
  \printlist{language}%
  \newunit\newblock
  \usebibmacro{byauthor}%
  \newunit\newblock
  \usebibmacro{byeditor+others}%
  \newunit\newblock
  \printfield{version}%
  \newunit
  \printfield{note}%
  \newunit\newblock
  \printlist{organization}%
  \newunit\newblock
  \iftoggle{bbx:eprint}
    {\usebibmacro{eprint}}
    {}%
  \newunit\newblock
  \usebibmacro{date}%
  \newunit\newblock
  \usebibmacro{url+urldate}%
  \newunit\newblock
  \usebibmacro{addendum+pubstate}%
  \setunit{\bibpagerefpunct}\newblock
  \usebibmacro{pageref}%
  \usebibmacro{finentry}}

\makeatletter
\DeclareFieldFormat{eprint:arxiv}{%
  arXiv\addcolon\linebreak[2]%
  \ifhyperref
    {\href{http://arxiv.org/\abx@arxivpath/#1}{%
       \nolinkurl{#1}%
       \iffieldundef{eprintclass}
	 {}
	 {\addspace\texttt{\mkbibbrackets{\thefield{eprintclass}}}}}}
    {\nolinkurl{#1}
     \iffieldundef{eprintclass}
       {}
       {\addspace\texttt{\mkbibbrackets{\thefield{eprintclass}}}}}}
\makeatother

\DeclareBibliographyDriver{inproceedings}{%
  \usebibmacro{bibindex}%
  \usebibmacro{begentry}%
  \usebibmacro{author/translator+others}%
  \setunit{\labelnamepunct}\newblock
  \usebibmacro{title}%
  \newunit
  \printlist{language}%
  \newunit\newblock
  \usebibmacro{byauthor}%
  \newunit\newblock
  \usebibmacro{in:}%
  \usebibmacro{maintitle+booktitle}%
  \newunit\newblock
  \usebibmacro{event+venue+date}%
  \newunit\newblock
  \usebibmacro{byeditor+others}%
  \newunit\newblock
  \iffieldundef{maintitle}
    {\printfield{volume}%
     \printfield{part}}
    {}%
  \newunit
  \printfield{volumes}%
  \newunit\newblock
  \usebibmacro{series+number}%
  \newunit\newblock
  \printfield{note}%
  \newunit\newblock
  \printlist{organization}%
  \newunit
  \usebibmacro{chapter+pages}%
  \newunit\newblock
  \usebibmacro{publisher+location+date}%
  \newunit\newblock
  \iftoggle{bbx:isbn}
    {\printfield{isbn}}
    {}%
  \newunit\newblock
  \usebibmacro{doi+eprint+url}%
  \newunit\newblock
  \usebibmacro{addendum+pubstate}%
  \setunit{\bibpagerefpunct}\newblock
  \usebibmacro{pageref}%
  \usebibmacro{finentry}}

\DeclareBibliographyDriver{inbook}{%
  \usebibmacro{bibindex}%
  \usebibmacro{begentry}%
  \usebibmacro{author/translator+others}%
  \setunit{\labelnamepunct}\newblock
  \usebibmacro{title}%
  \newunit
  \printlist{language}%
  \newunit\newblock
  \usebibmacro{byauthor}%
  \newunit\newblock
  \usebibmacro{in:}%
  \usebibmacro{bybookauthor}%
  \newunit\newblock
  \usebibmacro{maintitle+booktitle}%
  \newunit\newblock
  \usebibmacro{byeditor+others}%
  \newunit\newblock
  \printfield{edition}%
  \newunit
  \iffieldundef{maintitle}
    {\printfield{volume}%
     \printfield{part}}
    {}%
  \newunit
  \printfield{volumes}%
  \newunit\newblock
  \usebibmacro{series+number}%
  \newunit\newblock
  \printfield{note}%
  \newunit\newblock
  \usebibmacro{chapter+pages}%
  \newunit\newblock
  \usebibmacro{publisher+location+date}%
  \newunit\newblock
  \iftoggle{bbx:isbn}
    {\printfield{isbn}}
    {}%
  \newunit\newblock
  \usebibmacro{doi+eprint+url}%
  \newunit\newblock
  \usebibmacro{addendum+pubstate}%
  \setunit{\bibpagerefpunct}\newblock
  \usebibmacro{pageref}%
  \usebibmacro{finentry}}

\title{\bf Comparison of stratified-algebraic and topological K-theory}
\date{}

\author{Wojciech Kucharz\thanks{The first author was partially supported by
the National Science Centre (Poland), under grant number
2014/15/B/ST1/00046. He also acknowledges with gratitude support and
hospitality of the Max--Planck--Institut f\"ur Mathematik in Bonn.} \and Krzysztof
Kurdyka\thanks{The second author was partially supported by ANR (France)
grant STAAVF.}}

\address{Wojciech Kucharz\\Institute of Mathematics\\Faculty of Mathematics and Computer
Science\\Jagiellonian University\\ul.~\L{}ojasiewicza 6\\30-348
Krak\'ow\\Poland}
\email{Wojciech.Kucharz@im.uj.edu.pl}

\address{Krzysztof Kurdyka\\Laboratoire de Math\'ematiques\\UMR 5175 du
CNRS\\Universit\'e de Savoie\\Campus Scientifique\\73 376 Le
Bourget-du-Lac Cedex\newline France}
\email{kurdyka@univ-savoie.fr}
\usepackage[pdftex, pdfauthor={\theauthor}, pdftitle={\thetitle}]{hyperref}

\begin{document}
\maketitle
\thispagestyle{empty}

\begin{abstract}
Stratified-algebraic vector bundles on real algebraic varieties have
many desirable features of algebraic vector bundles but are more
flexible. We give a characterization of the compact real algebraic
varieties $X$ having the following property: There exists a positive
integer $r$ such that for any topological vector bundle $\xi$ on $X$,
the direct sum of $r$ copies of $\xi$ is isomorphic to a
stratified-algebraic vector bundle. In particular, each compact real
algebraic variety of dimension at most $8$ has this property. Our
results are expressed in terms of K-theory.
\end{abstract}



\section{Introduction and main results}\label{sec-1}

In the recent paper \cite{bib30}, we introduced and investigated
stratified-algebraic vector bundles on real algebraic varieties.
They occupy an intermediate position between algebraic and
topological vector bundles. Here we continue the line of research
undertaken in \cite{bib30, bib29} and look for new relationships
between stratified-algebraic and topological vector bundles. In a
broader context, the present paper is also closely related to
\cite{bib5, bib16, bib23, bib24, bib26, bib27, bib28}. All
results announced in this section are proved in
Section~\ref{sec-2}.

Throughout this paper the term \emph{real algebraic variety}
designates a locally ringed space isomorphic to an algebraic
subset of $\R^N$, for some $N$, endowed with the Zariski topology
and the sheaf of real-valued regular functions (such an object is
called an affine real algebraic variety in \cite{bib7}). The
class of real algebraic varieties is identical with the class of
quasi-projective real varieties, cf. \cite[Proposition~3.2.10,
Theorem~3.4.4]{bib7}. Morphisms of real algebraic varieties are
called \emph{regular maps}. Each real algebraic variety carries
also the Euclidean topology, which is induced by the usual metric
on $\R$. Unless explicitly stated otherwise, all topological
notions relating to real algebraic varieties refer to the
Euclidean topology.

Let $\F$ stand for $\R$, $\CB$ or $\HB$ (the quaternions). All
$\F$-vector spaces will be left $\F$-vector spaces. When
convenient, $\F$ will be identified with $\R^{d(\F)}$, where
\begin{equation*}
d(\F) = \dim_{\R} \F.
\end{equation*}

Let $X$ be a real algebraic variety. For any nonnegative integer
$n$, let $\varepsilon_X^n(\F)$ denote the standard trivial
$\F$-vector bundle on $X$ with total space $X \times \F^n$, where
$X \times \F^n$ is regarded as a real algebraic variety. An
algebraic $\F$-vector bundle on $X$ is an algebraic $\F$-vector
subbundle of $\varepsilon_X^n(\F)$ for some $n$ (cf.
\cite[Chapters~12 and~13]{bib7} for various characterizations of
algebraic $\F$-vector bundles).

We now recall the fundamental notion introduced in \cite{bib30}.
By a \emph{stratification} of $X$ we mean a finite collection
$\SC$ of pairwise disjoint Zariski locally closed subvarieties
whose union is $X$. Each subvariety in $\SC$ is called a stratum
of $\SC$. A \emph{stratified-algebraic $\F$-vector bundle on $X$}
is a topological $\F$-vector subbundle $\xi$ of
$\varepsilon_X^n(\F)$, for some $n$, such that for some
stratification $\SC$ of $X$, the restriction $\xi|_S$ of $\xi$ to
each stratum $S$ of $\SC$ is an algebraic $\F$-vector subbundle
of $\varepsilon_S^n(\F)$.

A topological $\F$-vector bundle $\xi$ on $X$ is said to
\emph{admit an algebraic structure} if it is isomorphic to an
algebraic $\F$-vector bundle on $X$. Similarly, $\xi$ is said to
\emph{admit a stratified-algebraic structure} if it is isomorphic
to a stratified-algebraic $\F$-vector bundle on $X$. These two
types of $\F$-vector bundles have been extensively investigated
in \cite{bib3, bib4, bib6, bib7, bib9, bib10, bib11, bib13} and
\cite{bib30, bib29}, respectively. In general, their behaviors
are quite different, cf. \cite[Example~1.11]{bib30}. Here we
further develop the direction of research initiated in
\cite{bib30, bib29}. It is convenient to bring into play
Grothendieck groups.

Denote by $\KF(X)$ the Grothendieck group of topological
$\F$-vector bundles on $X$. For any topological $\F$-vector
bundle $\xi$ on $X$, let $\llbracket \xi \rrbracket$ denote its
class in $\KF(X)$. Since $X$ has the homotopy type of a compact
polyhedron \cite[pp.~217, 225]{bib7}, it follows that the abelian
group $\KF(X)$ is finitely generated (cf.
\cite[Exercise~III.7.5]{bib21} or the spectral sequence in
\cite{bib2, bib15}). Let $\KFstr(X)$ be the subgroup of $\KF(X)$
generated by the classes of all $\F$-vector bundles admitting a
stratified-algebraic structure.

If the variety $X$ is compact, then the group $\KFstr(X)$
contains complete information on $\F$-vector bundles on $X$
admitting a stratified-algebraic structure. More precisely, we
have the following.

\begin{theorem}[{\cite[Corollary~3.14]{bib30}}]\label{th-1-1}
Let $X$ be a compact real algebraic variety. A topological
$\F$-vector bundle $\xi$ on $X$ admits a stratified-algebraic
structure if and only if the class $\llbracket \xi \rrbracket$ is
in $\KFstr(X)$.
\end{theorem}

In other words, with notation as in Theorem~\ref{th-1-1}, $\xi$
admits a stratified-algebraic structure if and only if there
exists a stratified-algebraic $\F$-vector bundle $\eta$ on $X$
such that the direct sum $\xi \oplus \eta$ admits a
stratified-algebraic structure.

For our purposes it is convenient to distinguish some vector
bundles by imposing a suitable condition on their rank. For any
topological $\F$-vector bundle $\xi$ on $X$, we regard $\rank
\xi$ (the rank of $\xi$) as a function
\begin{equation*}
\rank \xi \colon X \to \Z,
\end{equation*}
which assigns to every point $x$ in $X$ the dimension of the
fiber of $\xi$ over $x$. We say that $\xi$ has \emph{property}
\rk if for every integer $d$, the set
\begin{equation*}
\{ x \in X \mid (\rank \xi)(x) = d \}
\end{equation*}
is algebraically constructible. Recall that a subset of $X$ is
said to be algebraically constructible if it belongs to the
Boolean algebra generated by the Zariski closed subsets of $X$.
It readily follows that each stratified-algebraic $\F$-vector
bundle on $X$ has property \rk. Thus property \rk is a
necessary condition for $\xi$ to admit a stratified-algebraic
structure. Denote by $\KFrk(X)$ the subgroup of $\KF(X)$
generated by the classes of all topological $\F$-vector bundles
having property \rk. By construction,
\begin{equation*}
\KFstr(X) \subseteq \KFrk(X).
\end{equation*}
Since the group $\KF(X)$ is finitely generated, so is the
quotient group
\begin{equation*}
\GammaF(X) \coloneqq \KFrk(X) / \KFstr(X).
\end{equation*}
Thus the group $\GammaF(X)$ is finite if and only if
\begin{equation*}
r\GammaF(X) = 0
\end{equation*}
for some positive integer $r$. In the present paper the group $\GammaF(X)$
is the main object of investigation.

For any $\F$-vector bundle $\xi$ on $X$ and any positive integer
$r$, we denote by
\begin{equation*}
\xi(r) = \xi \oplus \cdots \oplus \xi
\end{equation*}
the $r$-fold direct sum. The following preliminary result shows
that our approach here is consistent with that of \cite{bib29}.

\begin{proposition}\label{prop-1-2}
Let $X$ be a compact real algebraic variety. For a positive
integer $r$, the following conditions are equivalent:
\begin{conditions}
\item\label{prop-1-2-a} The group $\GammaF(X)$ is finite and
$r\GammaF(X) = 0$.

\item\label{prop-1-2-b} For each topological $\F$-vector bundle
$\xi$ on $X$ having property \rk, the $\F$-vector bundle $\xi(r)$
admits a stratified-algebraic structure.

\item\label{prop-1-2-c} For each topological $\F$-vector bundle
$\eta$ on $X$ having constant rank, the $\F$-vector bundle
$\eta(r)$ admits a stratified-algebraic structure.
\end{conditions}
\end{proposition}

In \cite[Conjecture~C]{bib29}, it is suggested that the group
$\GammaF(X)$ is always finite (for $X$ compact). We show here
that the finiteness of the group $\GammaF(X)$ is equivalent to a
certain condition involving cohomology classes of a special kind.
For any nonnegative integer $k$, we defined in \cite{bib30} a
subgroup $\HCstr^{2k}(X; \Z)$ of the cohomology group $H^{2k}(X;
\Z)$. For the convenience of the reader, the definition
and basic properties of $\HCstr^{2k} (X; \Z)$ are
recalled in Section~\ref{sec-2}.

\begin{theorem}\label{th-1-3}
For any compact real algebraic variety $X$, the
following conditions are equivalent:
\begin{conditions}
\item\label{th-1-3-a} The group $\GammaF(X)$ is finite.

\item\label{th-1-3-b} The quotient group $H^{4k}(X; \Z)
/ \HCstr^{4k}(X; \Z)$ is finite for every positive
integer $k$ satisfying $8k -2 < \dim X$.
\end{conditions}
\end{theorem}

Since the groups $\HCstr^{2k}(-; \Z)$ are hard to
compute, it is worthwhile to give a simple topological
criterion for the finiteness of the group $\GammaF(X)$.
To this end some preparation is required.

For any positive integer $d$, let $\SB^d$ denote the
unit $d$-sphere
\begin{equation*}
\SB^d = \{ (u_0, \ldots, u_d) \in \R^{d+1} \mid u_0^2 +
\cdots + u_d^2 = 1 \}.
\end{equation*}
Let $s_d$ be a generator of the cohomology group
$H^d(\SB^d; \Z) \cong \Z$. A cohomology class $u$ in
$H^d(\Omega; \Z)$, where $\Omega$ is an arbitrary
topological space, is said to be \emph{spherical} if
$u = h^*(s_d)$ for some continuous map $h \colon \Omega
\to \SB^d$. Denote by $\Hsph^d(\Omega; \Z)$ the
subgroup of $H^d (\Omega; \Z)$ generated by all
spherical cohomology classes. In general a cohomology
class in $\Hsph^d(\Omega; \Z)$ need not be spherical.

\begin{theorem}\label{th-1-4}
Let $X$ be a compact real algebraic variety. If the
quotient group 
\begin{equation*}
H^{4k}(X; \Z) / \Hsph^{4k}(X; \Z)
\end{equation*}
is finite for every positive integer $k$ satisfying $8k -
2 < \dim X$, then the group $\GammaF(X)$ is finite.
\end{theorem}

As a consequence we obtain the following.

\begin{corollary}\label{cor-1-5}
Let $X$ be a compact real algebraic variety. If each
connected component of $X$ is homotopically equivalent
to $\SB^{d_1} \times \cdots \times \SB^{d_n}$ for some
positive integers $d_1, \ldots, d_n$, then the group
$\GammaF(X)$ is finite.
\end{corollary}

\begin{proof}
Since $\Hsph^l(X; \Z) = H^l(X; \Z)$ for every positive
integer $l$, it suffices to make use of
Theorem~\ref{th-1-4}.
\end{proof}

It is interesting to compare Corollary~\ref{cor-1-5}
with related, previously known, results. If ${X = X_1
\times \cdots \times X_n}$, where each $X_i$ is a
compact real algebraic variety homotopically equivalent
to $\SB^{d_i}$ for $1 \leq i \leq n$, then ${\GammaF(X) =
0}$ for ${\F = \CB}$ and ${\F = \HB}$, and ${2 \Gamma_{\R}(X)
= 0}$, cf. \cite[Theorem~1.10]{bib30}. On the other
hand, there exists a nonsingular real algebraic variety
$X$ diffeomorphic to the $n$-fold product $\SB^1 \times \cdots \times
\SB^1$, $n > d(\F)$, such that $\GammaF(X) \neq
0$, cf. \cite[Example~7.10]{bib30}.

For any compact real algebraic variety $X$, the equality
$H^l(X; \Z) = 0$ holds if $l > \dim X$, cf.
\cite[p.~217]{bib7}. Hence, in view of either
Theorem~\ref{th-1-3} or Theorem~\ref{th-1-4}, the group
$\GammaF(X)$ is finite for $\dim X \leq 6$. This is
extended below to $\dim X \leq 8$. Actually, we obtain
a result containing additional information.

Denote by $e(\F)$ the integer satisfying $d(\F) =
2^{e(\F)}$, that is,
\begin{equation*}
e(\F) =
\begin{cases}
0 & \textrm{if } \F = \R \\
1 & \textrm{if } \F = \CB \\
2 & \textrm{if } \F = \HB.
\end{cases}
\end{equation*}
Given a nonnegative integer $n$, set
\begin{align*}
a(n) &= \min \{ l \in \Z \mid l \geq 0,\ 2^l \geq n \},
\\
a(n, \F) &= \max \{ 0, a(n) - e(\F) \}.
\end{align*}
It is conjectured in \cite{bib29} that
\begin{equation*}
2^{a( \dim X, \F)} \GammaF(X) = 0
\end{equation*}
for every compact real algebraic variety $X$. This
conjecture is confirmed in \cite{bib29} for varieties
of dimension not exceeding $5$. Using different
methods, we get the following.

\begin{theorem}\label{th-1-6}
For any compact real algebraic variety $X$ of dimension
at most $8$, the group $\GammaF(X)$ is finite and
\begin{equation*}
2^{a (\dim X, \F) + a(X)} \GammaF(X) = 0,
\end{equation*}
where $a(X) = 0$ if $\dim X \leq 7$ and $a(X) = 2$ if
$\dim X = 8$.
\end{theorem}

We are not able to decide whether Theorem~\ref{th-1-6}
holds with $a(X)=0$ for $\dim X = 8$.

In Section~\ref{sec-2} we establish relationships
between the groups $\Hsph^{2k}(-;\Z)$ and
$\HCstr^{2k}(-; \Z)$ for $k \geq 1$. This leads to the
proofs of Theorems~\ref{th-1-3} and~\ref{th-1-4}. Along
the way we obtain closely related results,
Theorems~\ref{th-2-14}, \ref{th-2-15} and
\ref{th-2-16}, which are of independent interest.
Noteworthy is also Theorem~\ref{th-2-13}, which plays a
key role in the proof of Theorem~\ref{th-1-6}. In
Section~\ref{sec-3} we investigate topological
$\CB$-line bundles admitting a stratified-algebraic
structure.

\begin{notation}
Given two $\F$-vector bundles $\xi$ and $\eta$ on the
same topological space, we will write $\xi \cong \eta$
to indicate that they are isomorphic.
\end{notation}

\section{Stratified-algebraic versus topological vector
bundles}\label{sec-2}

To begin with we establish a connection between vector
bundles having property \rk and those of constant rank.

\begin{lemma}\label{lem-2-1}
Let $X$ be a real algebraic variety and let $\xi$ be a
topological $\F$-vector bundle on $X$. If $\xi$ has
property \rk, then there exists a stratified-algebraic
$\F$-vector bundle $\eta$ on $X$ such that the direct
sum $\xi \oplus \eta$ is of constant rank.
\end{lemma}

\begin{proof}
Since $X$ has the homotopy type of a compact polyhedron
\cite[pp.~217, 225]{bib7}, we may assume that $\xi$ is a
topological $\F$-vector subbundle of
$\varepsilon_X^n(\F)$ for some positive integer $n$.
Assume that $\xi$ has property \rk. By definition, for
each integer $d$ satisfying $0 \leq d \leq n$, the set
\begin{equation*}
R(d) = \{ x \in X \mid (\rank \xi)(x) = d \}
\end{equation*}
is algebraically constructible. Thus $R(d)$ is the
union of a finite collection of pairwise disjoint
Zariski locally closed subvarieties of $X$. In
particular, there exists a stratification $\SC$ of $X$
such that each set $R(d)$ is the union of some strata of
$\SC$. Furthermore, each nonempty set $R(d)$ is the
union of some connected components of $X$. It follows
that we can find a topological $\F$-vector subbundle
$\eta$ of $\varepsilon_X^n(\F)$ whose restriction
$\eta|_{R(d)}$ is the trivial $\F$-vector subbundle of
$\varepsilon_{R(d)}^n(\F)$ with total space $R(d)
\times (\F^{n-d} \times \{ 0 \})$, where $\F^{n-d}
\times \{ 0 \} \subseteq \F^n$. By construction, $\eta$
is a stratified-algebraic $\F$-vector bundle and the
direct sum $\xi\oplus\eta$ is of rank $n$.
\end{proof}

In particular, if $\KFcrk(X)$ is the subgroup of
$\KF(X)$ generated by the classes of all topological
$\F$-vector bundles of constant rank, then
\begin{equation*}
\KFstr(X) + \KFcrk(X) = \KFrk(X).
\end{equation*}
Hence the group $\GammaF(X)$ is isomorphic to the
quotient group
\begin{equation*}
\KFcrk(X) / \KFstr(X) \cap \KFcrk(X).
\end{equation*}

\begin{proof}[Proof of Proposition~\ref{prop-1-2}]
Obviously, (\ref{prop-1-2-b}) implies
(\ref{prop-1-2-a}). According to Theorem~\ref{th-1-1},
(\ref{prop-1-2-a}) implies (\ref{prop-1-2-b}). Hence,
in view of Lemma~\ref{lem-2-1},
(\ref{prop-1-2-a}) and (\ref{prop-1-2-c}) are
equivalent.
\end{proof}

Let $X$ be a real algebraic variety. Let $\K$ be a
subfield of $\F$, where $\K$ (as $\F$) stands for $\R$,
$\CB$ or $\HB$. Any $\F$-vector bundle $\xi$ on $X$ can
be regarded as a $\K$-vector bundle, which is indicated
by $\xi_{\K}$. In particular, $\xi_{\K} = \xi$ if $\K =
\F$. Furthermore, $\xi_{\R} = (\xi_{\K})_{\R}$. If the
$\F$-vector bundle $\xi$ admits a stratified-algebraic
structure, then so does the $\K$-vector bundle
$\xi_{\K}$.

The following result will be frequently referred to.

\begin{theorem}\label{th-2-2}
Let $X$ be a compact real algebraic variety. A
topological $\F$-vector bundle $\xi$ on $X$ admits a
stratified-algebraic structure if and only if the
$\K$-vector bundle $\xi_{\K}$ admits a
stratified-algebraic structure.
\end{theorem}

\begin{proof}
The proof for $\K=\R$, rather involved, is given
in \cite[Theorem~1.7]{bib30}. The general case follows
since $\xi_{\R} = (\xi_{\K})_{\R}$.
\end{proof}

We will also make use of the extension of scalars
construction. Let $X$ be a real algebraic variety. Any
$\K$-vector bundle $\xi$ on $X$ gives rise to the
$\F$-vector bundle $\F \otimes \xi$ on $X$. Here ${\F
\otimes \xi = \xi}$ if $\K=\F$, $\CB \otimes \xi$ is the
complexification of $\xi$ if $\K=\R$, and $\HB \otimes
\xi$ is the quaternionization of $\xi$ if $\K=\R$ or
$\K=\CB$. If the $\K$-vector bundle $\xi$ admits a
stratified-algebraic structure, then so does the
$\F$-vector bundle $\F \otimes \xi$.

For any $\CB$-vector bundle $\xi$, let $\bar{\xi}$
denote the conjugate bundle, cf. \cite{bib31}. Note
that
\begin{equation*}
\bar{\xi}_{\R} \cong \xi_{\R}.
\end{equation*}
Furthermore, for the $\HB$-vector bundle $\HB \otimes \xi$, we have
\begin{equation*}
(\HB \otimes \xi)_{\CB} \cong \xi \oplus \bar{\xi}.
\end{equation*}

\begin{lemma}\label{lem-2-3}
Let $X$ be a compact real algebraic variety and let
$\xi$ be a topological $\CB$-vector bundle on $X$. For
any positive integer $q$, the $\HB$-vector bundle $(\HB
\otimes \xi)(q)$ admits a stratified-algebraic
structure if and only if so does the $\CB$-vector
bundle $\xi(2q)$.
\end{lemma}

\begin{proof}
Since
\begin{equation*}
( (\HB \otimes \xi)(q) )_{\CB} \cong (\HB \otimes
\xi)_{\CB} (q) \cong (\xi \oplus \bar{\xi})(q)
\end{equation*}
and
\begin{equation*}
( (\xi \oplus \bar{\xi})(q) )_{\R} \cong  (\xi_{\R} \oplus
\bar{\xi}_{\R})(q) \cong  (\xi_{\R} \oplus \xi_{\R})(q) \cong
(\xi(2q))_{\R},
\end{equation*}
we get
\begin{equation*}
( (\HB \otimes \xi)(q) )_{\R} \cong \xi(2q))_{\R}.
\end{equation*}
The proof is complete in view of Theorem~\ref{th-2-2}.
\end{proof}

For any $\R$-vector bundle $\xi$, we have
\begin{equation*}
(\CB \otimes \xi)_{\R} \cong \xi \oplus \xi.
\end{equation*}

\begin{lemma}\label{lem-2-4}
Let $X$ be a compact real algebraic variety and let
$\xi$ be a topological $\R$-vector bundle on $X$. For
any positive integer $q$, the $\CB$-vector bundle $(\CB
\otimes \xi)(q)$ admits a stratified-algebraic
structure if and only if so does the $\R$-vector bundle
$\xi(2q)$.
\end{lemma}

\begin{proof}
Since
\begin{equation*}
( (\CB \otimes \xi)(q) )_{\R} \cong (\CB \otimes
\xi)_{\R}(q) \cong (\xi \oplus \xi)(q) \cong \xi(2q),
\end{equation*}
the proof is complete in view of Theorem~\ref{th-2-2}.
\end{proof}

For the convenience of the reader we recall the
definition and basic properties of
stratified-$\CB$-algebraic cohomology classes,
introduced and investigated in \cite{bib30}.

Let $V$ be a compact nonsingular real algebraic
variety. A \emph{nonsingular projective complexification} of
$V$ is a pair $(\V, \iota)$, where $\V$ is a
nonsingular projective scheme over $\R$ and $\iota
\colon V \to \V(\CB)$ is an injective map such that
$\V(\R)$ is Zariski dense in $\V$, $\iota(V) = \V(\R)$
and $\iota$ induces a biregular isomorphism between $V$
and $\V(\R)$. Here the set $\V(\R)$ of real points of
$\V$ is regarded as a subset of the set $\V(\CB)$ of
complex points of $\V$. The existence of $(\V,\iota)$
follows form Hironaka's theorem on resolution of
singularities \cite{bib19} (cf. also \cite{bib22} for a
very readable exposition). We identify $\V(\CB)$ with
the set of complex points of the scheme $\V_{\CB}
\coloneqq \V \times_{\Spec \R} \Spec \CB$ over $\CB$.
For any nonnegative integer $k$, denote by
$\Halg^{2k}(\V(\CB); \Z)$ the subgroup of
$H^{2k}(\V(\CB); \Z)$ that consists of the cohomology
classes corresponding to algebraic cycles (defined over
$\CB$) on $\V_{\CB}$ of codimension $k$, cf.
\cite{bib14} or \cite[Chapter~19]{bib17}. The subgroup
\begin{equation*}
\HCalg^{2k} (V; \Z) \coloneqq \iota^* (
\Halg^{2k}(\V(\CB); \Z) )
\end{equation*}
of $H^{2k}(V; \Z)$ does not depend on the choice of
$(\V; \iota)$, cf. \cite{bib6}. Cohomology classes in
$\HCalg^{2k}(V; \Z)$ are called \emph{$\CB$-algebraic}.
The groups $\HCalg^{2k}(-;\Z)$ are subtle invariants
with numerous applications, cf. \cite{bib6, bib8,
bib11, bib13, bib25}.

Let $X$ and $Y$ be real algebraic varieties. A map $f
\colon X \to Y$ is said to be \emph{stratified-regular}
if it is continuous and for some stratification $\SC$
of $X$, the restriction $f|_S \colon S \to Y$ of $f$ to
each stratum $S$ of $\SC$ is a regular map. A
cohomology class $u$ in $H^{2k}(X; \Z)$ is said to be
\emph{stratified-$\CB$-algebraic} if there exists a
stratified-regular map $\varphi \colon X \to
V$, into a compact nonsingular real algebraic variety $V$,
such that $u = \varphi^*(v)$ for some cohomology
class $v$ in $\HCalg^{2k}(V; \Z)$. The set
$\HCstr^{2k}(X; \Z)$ of all stratified-$\CB$-algebraic
cohomology classes in $H^{2k}(X; \Z)$ forms a subgroup.
The direct sum
\begin{equation*}
\HCstreven(X;\Z) \coloneqq \bigoplus_{k\geq 0}
\HCstr^{2k}(X;\Z)
\end{equation*}
is a subring of the ring
\begin{equation*}
\Heven(X;\Z) \coloneqq \bigoplus_{k\geq 0}
H^{2k}(X;\Z).
\end{equation*}
If $\xi$ is a stratified-algebraic $\CB$-vector bundle
on $X$, then the $k$th Chern class $c_k(\xi)$ of $\xi$
is in $\HCstr^{2k}(X;\Z)$ for every nonnegative integer
$k$. The reader can find proofs of these facts in
\cite{bib30}.

For any topological $\F$-vector bundle $\xi$ on $X$,
one can interpret $\rank \xi$ as an element of
$H^0(X;\Z)$. Then the following holds.

\begin{lemma}\label{lem-2-5}
Let $X$ be a real algebraic variety and let $\xi$ be a
topological $\F$-vector bundle on $X$. If $\xi$ has
property \rk, then $\rank \xi$ is in $\HCstr^0(X;\Z)$.
\end{lemma}

\begin{proof}
Assume that the $\F$-vector bundle $\xi$ has property
\rk. We make use of the notation introduced in the
proof of Lemma~\ref{lem-2-1}. Furthermore, we regard $V
= \{ 0, \ldots, n \}$ as a real algebraic variety and
$\rank \xi$ as a map
\begin{equation*}
\rank \xi \colon X \to V.
\end{equation*}
Then $\rank \xi$ is a stratified-regular map. Note that
$\rank \xi$ interpreted as a cohomology class in
$H^0(X; \Z)$ coincides with $(\rank\xi)^*(v)$, where $v$
is the cohomology class in $H^0(V;\Z)$ whose
restriction to the singleton $\{i\}$ is equal to $1$ in
$H^0(\{i\};\Z)$ for every $i$ in $V$. Since
$\HCalg^0(V;\Z) = H^0(V;\Z)$, the cohomology class
$(\rank\xi)^*(v)$ is in $\HCstr^0(X;\Z)$, as required.
\end{proof}

The following observation will prove to be useful.

\begin{proposition}\label{prop-2-6}
Let $X$ be a compact real algebraic variety. For a
topological $\CB$-vector bundle $\xi$ on $X$, the
following conditions are equivalent:
\begin{conditions}
\item\label{prop-2-6-a} There exists a positive integer
$r$ such that the $\CB$-vector bundle $\xi(r)$ admits a
stratified-algebraic structure.

\item\label{prop-2-6-b} The $\CB$-vector bundle $\xi$
has property \rk and for every positive integer $j$,
there exists a positive integer $b_j$ such that the
cohomology class $b_jc_j(\xi)$ is in
$\HCstr^{2j}(X;\Z)$.
\end{conditions}
\end{proposition}

\begin{proof}
Assume that condition (\ref{prop-2-6-a}) is satisfied.
Then $\xi(r)$ has property \rk and hence $\xi$ has it
as well. Furthermore, the total Chern class $c(\xi(r))$
is in $\HCstreven(X;\Z)$. We have
\begin{equation*}
c(\xi(r)) = c(\xi) \cupproduct \cdots \cupproduct
c(\xi),
\end{equation*}
where the right-hand-side is the $r$-fold cup product.
In particular, $c_1(\xi(r)) = rc_1(\xi)$ is in
$\HCstr^2(X;\Z)$. By induction, for every positive
integer $j$, we can find a positive integer $b_j$ such
that the cohomology class $b_jc_j(\xi)$ is in
$\HCstr^{2j}(X;\Z)$. Thus (\ref{prop-2-6-a}) implies
(\ref{prop-2-6-b}).

Now assume that condition (\ref{prop-2-6-b}) is
satisfied. Since $\xi$ has property \rk, by
Lemma~\ref{lem-2-5},~$\rank\xi$ is in
$\HCstr^0(X;\Z)$. Hence (\ref{prop-2-6-b}) implies that
the Chern character $\ch(\xi)$ is in ${\HCstreven(X;\Z)
\otimes_{\Z} \Q}$.\linebreak Consequently, for some positive
integer $r$, the class $r \llbracket \xi \rrbracket =
\llbracket \xi(r) \rrbracket$ is in $\KCstr(X)$, cf.
\cite[Proposition~8.9]{bib30}. According to
Theorem~\ref{th-1-1}, the $\CB$-vector bundle $\xi(r)$
admits a stratified-algebraic structure. Thus
(\ref{prop-2-6-b}) implies (\ref{prop-2-6-a}), which
completes the proof.
\end{proof}

We now collect some results on spherical cohomology
classes. Every compact real algebraic variety is
triangulable \cite[p.~217]{bib7} and hence a result due
to Serre can be stated as follows.

\begin{proposition}[{\cite[p.~289,
Propoposition~$2'$]{bib32}}]\label{prop-2-7}
Let $X$ be a compact real algebraic variety. Then there
exists a positive integer $a$ such that for every
positive integer $d$ satisfying
\begin{equation*}
\dim X \leq 2d-2
\end{equation*}
and every cohomology class $u$ in $H^d(X;\Z)$, the
cohomology class $au$ is spherical. In particular, the
inclusion
\begin{equation*}
aH^d(X;\Z) \subseteq \Hsph^d(X;\Z)
\end{equation*}
holds for such $a$ and $d$.
\end{proposition}

Let $X$ and $Y$ be real algebraic varieties. A map $f
\colon X \to Y$ is said to be \emph{continuous
rational} if it is continuous and its restriction to
some Zariski open and dense subvariety of $X$ is a
regular map. Assuming that the variety $X$ is
nonsingular, the map $f$ is continuous rational if and
only if it is stratified-regular, cf.
\cite[Proposition~8]{bib23} and
\cite[Remark~2.3]{bib30}.

\begin{lemma}\label{lem-2-8}
Let $X$ be a compact nonsingular real algebraic variety
and let $d$ be a positive integer. For any continuous
map $h \colon X \to \SB^d$ and any continuous map
$\varphi \colon \SB^d \to \SB^d$ of (topological)
degree $2$, the composite map $\varphi \circ h \colon X
\to \SB^d$ is homotopic to a stratified-regular map.
\end{lemma}

\begin{proof}
We may assume without loss of generality that $h$ is a
$\C^{\infty}$ map. By Sard's theorem, $h$ is transverse
to each point in some open subset $U$ of $\SB^d$
diffeomorphic to $\R^d$. Let $y$ and $z$ be distinct
points in $U$, and let $A$ be a $\C^{\infty}$ arc in
$U$ joining $y$ and $z$. Then
\begin{equation*}
M \coloneqq h^{-1}(y) \cup h^{-1}(z)
\end{equation*}
is a compact $\C^{\infty}$ submanifold of $X$.
Furthermore, $B \coloneqq h^{-1}(A)$ is a compact
$\C^{\infty}$ manifold with boundary $\partial B = M$,
embedded in $X$ with trivial normal bundle. Hence,
according to \cite[Theorem~1.12]{bib12}, there exists a
$\C^{\infty}$ map $F \colon X \to \R^d$ transverse to
$0$ in $\R^d$ and such that
\begin{equation*}
M = F^{-1}(0).
\end{equation*}
By the Weierstrass approximation theorem, the
$\C^{\infty}$ map $F$ can be approximated, in the
$\C^{\infty}$ topology, by a regular map $G \colon X
\to \R^d$. If $G$ is sufficiently close to $F$, then
$G$ is transverse to $0$ and
\begin{equation*}
V \coloneqq G^{-1}(0)
\end{equation*}
is a nonsingular Zariski closed subvariety of $X$.
Furthermore, $V$ is isotopic to $M$ in $X$, cf.
\cite[Theorem~20.2]{bib1}.

We can choose a $\C^{\infty}$ map $\psi \colon \SB^d
\to \SB^d$ of degree $2$ that is transverse to $y$ and
satisfies $\psi^{-1}(y) = \{y,z\}$. By Hopf's theorem,
$\psi$ is homotopic to $\varphi$. Consequently, the
maps $\varphi \circ h$ and $\psi \circ h$ are homotopic.
It suffices to prove that $\psi \circ h$ is homotopic
to a stratified-regular map. By construction, the map
$\psi \circ h$ is transverse to $y$ and
\begin{equation*}
(\psi \circ h)^{-1} (y) = h^{-1} (\psi^{-1}(y)) =
h^{-1}(y) \cup h^{-1}(z) = M.
\end{equation*}
Since $M$ is isotopic to $V$, according to
\cite[Theorem~2.4]{bib24}, the map $\psi \circ h$ is
homotopic to a continuous rational map $f \colon X \to
\SB^d$. The map $f$ is stratified-regular, the variety
$X$ being nonsingular.
\end{proof}

As a consequence, we obtain the following observation.

\begin{remark}\label{rem-2-9}
For any compact nonsingular real algebraic variety $X$,
the inclusion
\begin{equation*}
2 \Hsph^{2k}(X;\Z) \subseteq \HCstr^{2k}(X;\Z)
\end{equation*}
holds for every positive integer $k$. Indeed, it
suffices to prove that for any spherical cohomology
class $u$ in $H^{2k}(X;\Z)$, the cohomology class $2u$
is in $\HCstr^{2k}(X; \Z)$. To this end, let $h \colon X
\to \SB^{2k}$ be a continuous map with $h^*(s_{2k}) =
u$ and let $\varphi \colon \SB^{2k} \to \SB^{2k}$ be a
continuous map of degree $2$. Then
\begin{equation*}
(\varphi \circ h)^* (s_{2k}) = h^*( \varphi^*(s_{2k}) )
= h^*(2s_{2k}) = 2u.
\end{equation*}
Recall that $\HCalg^{2k} (\SB^{2k}; \Z) = H^{2k}
(\SB^{2k}; \Z)$, cf. \cite[Proposition~4.8]{bib6}.
Since, according to Lem\-ma~2.8, the map $\varphi \circ
h$ is homotopic to a stratified-regular map, it follows
that the cohomology class $2u$ is in
$\HCstr^{2k}(X;\Z)$.
\end{remark}

It would be interesting to decide whether the
nonsingularity of $X$ in Remark~\ref{rem-2-9} is
essential. Dropping the nonsingularity assumption, we
obtain below a weaker but useful result,
Lemma~\ref{lem-2-12}. First some preparation is
necessary.

By a \emph{multiblowup} of a real algebraic variety $X$
we mean a regular map $\pi \colon X' \to X$ which is
the composition of a finite collection of blowups with
nonsingular centers. If $C$ is a Zariski closed
subvariety of $X$ and the restriction $\pi_C \colon X'
\setminus \pi^{-1}(C) \to X \setminus C$ of $\pi$ is a
biregular isomorphism, then we say that the multiblowup
$\pi$ is \emph{over} $C$.

A \emph{filtration} of $X$ is a finite sequence $\FC =
(X_{-1}, X_0, \ldots, X_m)$ of Zariski closed
subvarieties satisfying
\begin{equation*}
\varnothing = X_{-1} \subseteq X_0 \subseteq \cdots
\subseteq X_m = X.
\end{equation*}

We will make use of the following result.

\begin{theorem}[{\cite[Theorem~5.4]{bib30}}]\label{th-2-10}
Let $X$ be a compact real algebraic variety. For a
topological $\F$-vector bundle $\xi$ on $X$, the
following conditions are equivalent:
\begin{conditions}
\item\label{th-2-10-a} The $\F$-vector bundle $\xi$
admits a stratified-algebraic structure.

\item\label{th-2-10-b} There exists a filtration $\FC =
(X_{-1}, X_0, \ldots, X_m)$ of $X$, and for each $i =
0, \ldots, m$, there exists a multiblowup $\pi_i \colon
X'_i \to X_i$ over $X_{i-1}$ such that the pullback
$\F$-vector bundle $\pi_i^* (\xi|_{X_i})$ on $X'_i$
admits a stratified-algebraic structure.
\end{conditions}
\end{theorem}

We now derive the following.

\begin{lemma}\label{lem-2-11}
Let $X$ be a compact real algebraic variety. Let $d$ be
a positive integer and let $\theta$ be a topological
$\F$-vector bundle on $\SB^d$. For any continuous map
$h \colon X \to \SB^d$ and any continuous map $\varphi
\colon \SB^d \to \SB^d$ of degree $2$, the pullback
$\F$-vector bundle $(\varphi \circ h)^*\theta$ on $X$
admits a stratified-algebraic structure.
\end{lemma}

\begin{proof}
Let $\FC = (X_{-1}, X_0, \ldots, X_m)$ be a filtration
of $X$ such that the variety $X_i \setminus X_{i-1}$ is
nonsingular for $0 \leq i \leq m$. According to
Hironaka's theorem on resolution of singularities
\cite{bib19, bib22}, for each $i = 0, \ldots, m$, there
exists a multiblowup $\pi_i \colon X'_i \to X_i$ over
$X_{i-1}$ with $X'_i$ nonsingular. In view of
Theorem~\ref{th-2-10}, the $\F$-vector bundle $\xi
\coloneqq (\varphi \circ h)^*\theta$ on $X$ admits a
stratified-algebraic structure if and only if the
$\F$-vector bundle $\xi_i \coloneqq
\pi_i^*(\xi|_{X_i})$ on $X'_i$ admits a
stratified-algebraic structure for $0 \leq i \leq m$.
If $e_i \colon X_i \hookrightarrow X$ is the inclusion
map, then
\begin{equation*}
\xi_i = \pi_i^* (e_i^*\xi) = \pi_i^* ( e_i^* ( (\varphi
\circ h)^* \theta ) ) = ( \varphi \circ h \circ e_i
\circ \pi_i)^* \theta.
\end{equation*}
Since the variety $X'_i$ is nonsingular and the map $h
\circ e_i \circ \pi_i \colon X'_i \to \SB^d$ is
continuous, according to Lemma~\ref{lem-2-8}, the map
$\varphi \circ h \circ e_i \circ \pi_i$ is homotopic to
a stratified-regular map ${f_i \colon X'_i \to \SB^d}$.
In particular, $\xi_i \cong f_i^*\theta$. We may assume
that the $\F$-vector bundle $\theta$ is algebraic since
each topological $\F$-vector bundle on $\SB^d$ admits
an algebraic structure, cf. \cite[Theorem~11.1]{bib34}
and \cite[Proposition~12.1.12; pp.~325, 326, 352]{bib7}.
Thus $f_i^*\theta$ is a stratified-algebraic
$\F$-vector bundle on $X'_i$. Consequently, the
$\F$-vector bundle $\xi_i$ admits a
stratified-algebraic structure, as required.
\end{proof}

Here is the result we have already alluded to in the
comment following Remark~\ref{rem-2-9}.

\begin{lemma}\label{lem-2-12}
For any compact real algebraic variety $X$, the
inclusion
\begin{equation*}
2 (k-1)! \Hsph^{2k} (X;\Z) \subseteq \HCstr^{2k}(X;\Z).
\end{equation*}
holds for every positive integer $k$.
\end{lemma}

\begin{proof}
Let $k$ be a positive integer. It suffices to prove that for every
spherical cohomology class $u$ in $H^{2k}(X;\Z)$, the cohomology class
$2(k-1)!u$ is in $\HCstr^{2k}(X;\Z)$. To this end, let $h \colon X \to
\SB^{2k}$ be a continuous map with $h^*(s_{2k})=u$ and let $\varphi
\colon \SB^{2k} \to \SB^{2k}$ be a continuous map of degree $2$. Then
\begin{equation*}
(\varphi\circ h)^* (s_{2k}) = h^*( \varphi^*(s_{2k}) ) = h^*(2s_{2k})
=2u.
\end{equation*}
Now we choose a topological $\CB$-vector bundle $\theta$ on $\SB^{2k}$
with
\begin{equation*}
c_k(\theta) = (k-1)!s_{2k},
\end{equation*}
cf. \cite[p.~19]{bib2} or \cite[p.~155]{bib18}. Then
\begin{equation*}
c_k ( (\varphi \circ h)^*\theta ) = (\varphi \circ \theta)^* (c_k
(\theta) ) = (\varphi \circ h)^* ( (k-1)! s_{2k} ) = 2(k-1)!u.
\end{equation*}
According to Lemma~\ref{lem-2-11}, the $\CB$-vector bundle $( \varphi
\circ h)^* \theta$ on $X$ admits a stratified-algebraic structure, and
hence the cohomology class $2(k-1)!u$ is in $\HCstr^{2k}(X;\Z)$.~The
proof is\linebreak complete.
\end{proof}

The following result will be used in the proof of Theorem~\ref{th-1-6}
and is also of independent interest.

\begin{theorem}\label{th-2-13}
Let $X$ be a compact real algebraic variety. Let $k$ be a positive
integer and let $\theta$ be a topological $\F$-vector bundle on
$\SB^{2k}$, where $\F=\CB$ or $\F=\HB$. For any continuous map $h \colon
X \to \SB^{2k}$, the $\F$-vector bundle $h^*\theta \oplus h^*\theta$ on
$X$ admits a stratified-algebraic structure.
\end{theorem}

\begin{proof}
Let $\varphi \colon \SB^{2k} \to \SB^{2k}$ be a continuous map of degree
$2$. Then
\begin{equation*}
c_k (\varphi^*\theta_{\CB}) = \varphi^*( c_k(\theta_{\CB}) ) =
2c_k(\theta_{\CB}) = c_k ( \theta_{\CB} \oplus \theta_{\CB} ),
\end{equation*}
and hence the $\CB$-vector bundles $\varphi^*\theta_{\CB}$ and
$\theta_{\CB} \oplus \theta_{\CB}$ on $\SB^{2k}$ are stably equivalent,
cf. \cite[p.~19]{bib2} or \cite[p.~155]{bib18}. Consequently, the
$\CB$-vector bundles
\begin{equation*}
h^*(\varphi^*\theta_{\CB}) = (\varphi \circ h)^* \theta_{\CB} \
\textrm{and} \ h^*(\theta_{\CB} \oplus \theta_{\CB}) = (h^* \theta
\oplus h^*\theta)_{\CB}
\end{equation*}
on $X$ are stably equivalent as well. By Lemma~\ref{lem-2-11}, the
$\CB$-vector bundle $(\varphi \circ h)^* \theta_{\CB}$ admits a
stratified-algebraic structure. Hence, according to
Theorem~\ref{th-1-1}, the $\CB$-vector bundle $(h^*\theta \oplus
h^*\theta)_{\CB}$ admits a stratified-algebraic structure. Now the proof
is complete in view of Theorem~\ref{th-2-2}.
\end{proof}

The next three theorems are crucial for the proof of
Theorem~\ref{th-1-3}. We first consider $\HB$-vector bundles. Note that
for any $\HB$-vector bundle $\xi$, we have
\begin{equation*}
c_l(\xi_{\CB})=0
\end{equation*}
for every odd positive integer $l$.

\begin{theorem}\label{th-2-14}
Let $X$ be a compact real algebraic variety. For a topological
$\HB$-vector bundle $\xi$ on $X$, the following conditions are
equivalent:
\begin{conditions}
\item\label{th-2-14-a} There exists a positive integer $r$ such that the
$\HB$-vector bundle $\xi(r)$ admits a stratified-algebraic structure.

\item\label{th-2-14-b} The $\HB$-vector bundle $\xi$ has property \rk
and there exists a positive integer $a$ such that the cohomology class
$ac_{2k}(\xi_{\CB})$ is in $\HCstr^{4k}(X;\Z)$ for every positive integer
$k$ satisfying $8k-2<\dim X$.
\end{conditions}
\end{theorem}

\begin{proof}
If condition (\ref{th-2-14-a}) is satisfied, then the $\CB$-vector
bundle $\xi_{\CB}(r)$ admits a stratified-algebraic structure, being
isomorphic to $(\xi(r))_{\CB}$. Thus condition (\ref{th-2-14-b}) holds
in view of Proposition~\ref{prop-2-6}.

Now assume that condition (\ref{th-2-14-b}) is satisfied. By
Proposition~\ref{prop-2-7} and Lemma~\ref{lem-2-12}, there exists a
positive integer $b$ such that the cohomology class $bc_{2k}(\xi_{\CB})$
is in $\HCstr^{4k}(X;\Z)$ for every positive integer $k$. Furthermore,
$c_l(\xi_{\CB})=0$ for every odd positive integer $l$. Hence, according
to Proposition~\ref{prop-2-6}, there exists a positive integer $r$ such
that the $\CB$-vector bundle $\xi_{\CB}(r)$ admits a
stratified-algebraic structure. Since the $\CB$-vector bundles
$\xi_{\CB}(r)$ and $(\xi(r))_{\CB}$ are isomorphic, by
Theorem~\ref{th-2-2}, the $\HB$-vector bundle $\xi(r)$ admits a
stratified-algebraic structure. Thus (\ref{th-2-14-b}) implies
(\ref{th-2-14-a}). The proof is complete.
\end{proof}

Recall that for any topological $\CB$-vector bundle $\xi$, the equality
\begin{equation*}
c_k(\bar{\xi}) = (-1)^k c_k(\xi)
\end{equation*}
holds for every nonnegative integer $k$, cf. \cite[p.~168]{bib31}.

\begin{theorem}\label{th-2-15}
Let $X$ be a compact real algebraic variety. For a topological
$\CB$-vector bundle $\xi$ on $X$, the following conditions are
equivalent:
\begin{conditions}
\item\label{th-2-15-a} There exists a positive integer $r$ such that
the $\CB$-vector bundle $\xi(r)$ admits a stratified-algebraic
structure.

\item\label{th-2-15-b} The $\CB$-vector bundle $\xi$ has property \rk
and there exists a positive integer $a$ such that the cohomology class
$ac_{2k}(\xi \oplus \bar{\xi})$ is in $\HCstr^{4k}(X;\Z)$ for every
positive integer $k$ satisfying $8k-2<\dim X$.
\end{conditions}
\end{theorem}

\begin{proof}
Since $(\HB \otimes \xi)_{\CB} \cong \xi \oplus \bar{\xi}$, the equality
\begin{equation*}
c_l( (\HB \otimes \xi)_{\CB} ) = c_l(\xi \oplus \bar{\xi})
\end{equation*}
holds for every nonnegative integer $l$. Furthermore, the $\CB$-vector
bundle $\xi$ has property \rk if and only if the $\HB$-vector bundle
$\HB \otimes \xi$ has it. Hence the proof is complete in view of
Lemma~\ref{lem-2-3} and Theorem~\ref{th-2-14}.
\end{proof}

Let $\xi$ be an $\R$-vector bundle. Recall that for any nonnegative
integer $k$, the $k$th Pontryagin class of $\xi$ is defined by
\begin{equation*}
p_k(\xi) = (-1)^k c_{2k}(\CB \otimes \xi).
\end{equation*}

\begin{theorem}\label{th-2-16}
Let $X$ be a compact real algebraic variety. For a topological
$\R$-vector bundle $\xi$ on $X$, the following conditions are
equivalent:
\begin{conditions}
\item\label{th-2-16-a} There exists a positive integer $r$ such that the
$\R$-vector bundle $\xi(r)$ admits a stratified-algebraic structure.

\item\label{th-2-16-b} The $\R$-vector bundle $\xi$ has property \rk
and there exists a positive integer $a$ such that the cohomology class
$ap_k(\xi)$ is in $\HCstr^{4k}(X;\Z)$ for every positive integer $k$
satisfying $8k-2<\dim X$.
\end{conditions}
\end{theorem}

\begin{proof}
Assume that condition (\ref{th-2-16-a}) is satisfied. Then the
$\R$-vector bundle $\xi(r)$ has property \rk and hence $\xi$ has it as
well. Furthermore, the $\CB$-vector bundle ($\CB \otimes \xi)(r)$ admits
a stratified-algebraic structure, being isomorphic to $\CB \otimes
\xi(r)$. According to Proposition~\ref{prop-2-6}, for every positive
integer $j$, there exists a positive integer $b_j$ such that the
cohomology class $b_jc_j(\CB \otimes \xi)$ is in $\HCstr^{2j}(X;\Z)$. In
particular, (\ref{th-2-16-a}) implies (\ref{th-2-16-b}) in view of the
definition of $p_k(\xi)$.

Now assume that condition (\ref{th-2-16-b}) is satisfied. By
Proposition~\ref{prop-2-7} and Lemma~\ref{lem-2-12}, there exists a
positive integer $b$ such that the cohomology class $bc_{2k}(\CB \otimes
\xi)$ is in $\HCstr^{4k}(X;\Z)$ for every positive integer $k$. Recall
that $2c_l(\CB \otimes \xi) =0$ for every odd positive integer $l$, cf.
\cite[p.~174]{bib31}. Hence, according to Proposition~\ref{prop-2-6},
the $\CB$-vector bundle $(\CB \otimes \xi)(q)$ admits a
stratified-algebraic structure for some positive integer $q$. In view of
Lemma~\ref{lem-2-4}, the $\R$-vector bundle $\xi(2q)$ admits a
stratified-algebraic structure. Thus (\ref{th-2-16-b}) implies
(\ref{th-2-16-a}). The proof is complete.
\end{proof}

We need one more technical result.

\begin{lemma}\label{lem-2-17}
Let $X$ be a compact real algebraic variety. If the group
$\Gamma_{\CB}(X)$ is finite, then the quotient group $H^{2j}(X;\Z) /
\HCstr^{2j}(X;\Z)$ is finite for every positive integer $j$. If the
group $\GammaF(X)$ is finite, where $\F=\R$ or $\F=\HB$, then the
quotient group $H^{4k}(X;\Z) / \HCstr^{4k}(X;\Z)$ is finite for every
positive integer $k$.
\end{lemma}

\begin{proof}
Recall that the cohomology group $H^*(X;\Z)$ is finitely generated, the
variety $X$ being triangulable.

There exists a positive integer $b$ such that for every positive integer
$j$ and every cohomology class $u$ in $H^{2j}(X;\Z)$, one can find a
topological $\CB$-vector bundle $\xi$ on $X$ with
\begin{equation*}
c_i(\xi)=0 \textrm{ for } 1 \leq i \leq j-1 \textrm{ and } c_j(\xi)=bu,
\end{equation*}
cf. \cite[p.~19]{bib2} or \cite[p.~155, Theorem~A]{bib18}. We can choose such a
$\CB$-vector bundle $\xi$ of constant rank.

Assume that the group $\Gamma_{\CB}(X)$ is finite and
$r\Gamma_{\CB}(X)=0$ for some positive integer $r$. Then the
$\CB$-vector bundle $\xi(r)$ admits a stratified-algebraic structure,
and hence the cohomology class
\begin{equation*}
c_j(\xi(r)) = rc_j(\xi) = rbu
\end{equation*}
is in $\HCstr^{2j}(X;\Z)$. Thus the quotient group $H^{2j}(X;\Z) /
\HCstr^{2j}(X;\Z)$ is finite, as asserted.

Note that the complexification $\CB \otimes \xi_{\R}$ of the $\R$-vector
bundle $\xi_{\R}$ satisfies
\begin{equation*}
\CB \otimes \xi_{\R} \cong \xi \oplus \bar{\xi}.
\end{equation*}
Similarly, for the quaternionization $\HB \otimes \xi$ of the
$\CB$-vector bundle $\xi$, we have
\begin{equation*}
(\HB \otimes \xi)_{\CB} \cong \xi \oplus \bar{\xi}.
\end{equation*}

If the group $\GammaR(X)$ is finite and $q\GammaR(X)=0$ for some
positive integer $q$, then the $\R$-vector bundle $\xi_{\R}(q)$ admits a
stratified-algebraic structure, and hence so do the $\CB$-vector bundles
\begin{equation*}
\CB \otimes \xi_{\R} (q) \cong (\CB \otimes \xi_{\R}) (q) \cong (\xi
\oplus \bar{\xi})(q).
\end{equation*}

If the group $\Gamma_{\HB}(X)$ is finite and $q\Gamma_{\HB}(X)=0$, then
the $\HB$-vector bundle $(\HB \otimes \xi)(q)$ admits a
stratified-algebraic structure, and hence so do the $\CB$-vector bundles
\begin{equation*}
( (\HB \otimes \xi) (q) )_{\CB} \cong (\HB \otimes \xi)_{\CB} (q) \cong (\xi
\oplus \bar{\xi})(q).
\end{equation*}

Consequently, if $q\GammaF(X)=0$, where $\F=\R$ or $\F=\HB$, then the
Chern class $c_j( (\xi \oplus \bar{\xi}) (q) )$ is in $\HCstr^{2j}
(X;\Z)$. Now suppose that $j=2k$, where $k$ is a positive integer. Then
\begin{equation*}
c_i(\xi \oplus \bar{\xi}) = 0 \textrm{ for } 1 \leq i \leq 2k-1
\textrm{ and } c_{2k}(\xi \oplus \bar{\xi}) = 2c_{2k}(\xi) = 2bu,
\end{equation*}
which implies the equality
\begin{equation*}
c_{2k} ( (\xi \oplus \bar{\xi}) (q) ) = qc_{2k} (\xi \oplus \bar{\xi})
=2qbu.
\end{equation*}
Thus the cohomology class $2qbu$ is in $\HCstr^{4k}(X;\Z)$. In
conclusion, the quotient group 
\begin{equation*}
H^{4k}(X;\Z) / \HCstr^{4k}(X;\Z)
\end{equation*}
is finite. The proof is complete.
\end{proof}

We are now ready to prove the theorems announced in Section~\ref{sec-1}.

\begin{proof}[Proof of Theorem~\ref{th-1-3}]
In view of Lemma~\ref{lem-2-17}, condition (\ref{th-1-3-a}) implies
(\ref{th-1-3-b}). By combining Theorems~\ref{th-2-14}, \ref{th-2-15}
and~\ref{th-2-16}, we conclude that (\ref{th-1-3-b}) implies
(\ref{th-1-3-a}).
\end{proof}

\begin{proof}[Proof of Theorem~\ref{th-1-4}]
It suffices to make use of Theorem~\ref{th-1-3} and
Lemma~\ref{lem-2-12}.
\end{proof}

\begin{proof}[Proof of Theorem~\ref{th-1-6}]
Let $n=\dim X$. According to Proposition~\ref{prop-1-2}, it suffices to
prove that for any topological $\F$-vector bundle $\xi$ of constant
positive rank on $X$, the $\F$-vector bundle $\xi(r)$ admits a
stratified-algebraic structure, where
\begin{equation*}
r =
\begin{cases}
2^{a(n, \F)} & \textrm{if } n\leq 7\\
2^{a(n,\F)+2} & \textrm{if } n=8.
\end{cases}
\end{equation*}
If $n \leq d(\F)$, then $a(n,\F)=1$ and the $\F$-vector bundle
$\xi(1)=\xi$ admits a stratified-algebraic structure, cf.
\cite[Corollary~3.6]{bib30}. Henceforth we assume that
\begin{equation*}
n \geq d(\F)+1.
\end{equation*}
The rest of the proof is divided into three steps.

\begin{case}\label{case-1}
Suppose that $\F=\HB$.
\end{case}

The $4$-sphere $\SB^4$ can be identified (as a topological space) with
the quaternionic projective line $\PB^1(\HB)$. Let $\theta$ be the
$\HB$-line bundle on $\SB^4$ corresponding to the tautological
$\HB$-line bundle on $\PB^1(\HB)$. Since $5 \leq n \leq 8$, we have
$a(n,\HB)=1$.

First suppose that $5 \leq n \leq 7$. Then $\xi$ can be expressed as
\begin{equation*}
\xi = \lambda \oplus \varepsilon,
\end{equation*}
where $\lambda$ and $\varepsilon$ are topological $\HB$-vector bundles,
$\rank \lambda =1$ and $\varepsilon$ is trivial, cf.
\cite[p.~99]{bib20}. For the same reason, the $\HB$-vector bundle
$\lambda \oplus \lambda$ has a nowhere vanishing continuous section.
Thus the $\HB$-line bundle $\lambda$ is generated by two continuous
sections. It follows that we can find a continuous map $h \colon X \to
\SB^4$ with $\lambda \cong h^*\theta$. According to
Theorem~\ref{th-2-13}, the $\HB$-vector bundle $\lambda \oplus \lambda =
\lambda(2)$ admits a stratified-algebraic structure. Since
\begin{equation*}
\xi(2) \cong \lambda(2) \oplus \varepsilon(2),
\end{equation*}
the $\HB$-vector bundle $\xi(2)$ admits a stratified-algebraic
structure, as required.

Now suppose that $n=8$. It remains to prove that the $\HB$-vector bundle
$\xi(8)$ admits a stratified-algebraic structure. This can be done as
follows. Let $\FC = (X_{-1}, X_0, \ldots, X_m)$ be a filtration of $X$
such that the variety $X_i \setminus X_{i-1}$ is nonsingular of pure
dimension for $0 \leq i \leq m$. According to Hironaka's theorem on
resolution of singularities \cite{bib19, bib22}, for each $i=0, \ldots,
m$, there exists a multiblowup $\pi_i \colon X'_i \to X_i$ over
$X_{i-1}$ with $X'_i$ nonsingular of pure dimension. Consider the
pullback $\HB$-vector bundle $\xi_i \coloneqq \pi_i^*(\xi|_{X_i})$ on
$X'_i$. According to Theorem~\ref{th-2-10}, it suffices to prove that
the $\HB$-vector bundle $\xi_i(8)$ admits a stratified-algebraic
structure. If $\dim X'_i \leq 7$, we already established a stronger
result, namely, $\xi_i(2)$ admits a stratified-algebraic structure. If
$\dim X'_i = 8$, we choose a finite subset $A_i$ of $X'_i$ whose
intersection with each connected component of $X'_i$ consists of one
point. Let $\sigma_i \colon X''_i \to X'_i$ be the blowup of $X'_i$
with center $A_i$. We can replace $\pi_i \colon X'_i \to X_i$ by the
composite map $\sigma_i \circ \pi_i \colon X''_i \to X_i$ and replace the
$\HB$-vector bundle $\xi_i$ on $X'_i$ by the $\HB$-vector bundle
$(\sigma_i
\circ \pi_i)^* (\xi|_{X_i})$ on $X''_i$. Note that $X''_i$ is a compact
nonsingular real algebraic variety of pure dimension $8$, and each
connected component of $X''_i$ is nonorientable as a $\C^{\infty}$
manifold. Thus in order to simplify notation we may assume that the
variety $X$ is nonsingular of pure dimension $8$, and each connected
component of $X$ is nonorientable as a $\C^{\infty}$ manifold. The last
condition implies the equality
\begin{equation*}
2H^8(X;\Z)=0.
\end{equation*}
Since $c_l(\xi_{\CB})=0$ for every odd positive integer $l$, we get
\begin{equation*}
c_4 ( (\xi(4))_{\CB} ) = c_4 ( \xi_{\CB}(4) ) = 4c_4(\xi_{\CB}) +
6c_2(\xi_{\CB}) \cupproduct c_2(\xi_{\CB}) =0
\end{equation*}
in $H^8(X;\Z)$. The $\HB$-vector bundle $\xi(4)$ can be expressed as the
direct sum of a topological $\HB$-vector bundle $\eta$ of rank $2$ and a
trivial $\HB$-vector bundle, cf. \cite[p.~99]{bib20}. Then
\begin{equation*}
c_4(\eta_{\CB}) = c_4( (\xi(4))_{\CB} ) = 0.
\end{equation*}
Recall that $c_4 (\eta_{\CB})$ is the Euler class $e(\eta_{\R})$ of the
oriented $\R$-vector bundle $\eta_{\R} = (\eta_{\CB})_{\R}$, cf.
\cite[p.~159]{bib31}. Interpreting $e(\eta_{\R})$ as an obstruction, we
conclude that the $\HB$-vector bundle $\eta$ has a nowhere vanishing
continuous section, cf. \cite[p.~139, 140, 147]{bib31} and \cite{bib33}.
Consequently, the $\HB$-vector bundle $\xi(4)$ can be expressed as
\begin{equation*}
\xi(4) = \mu \oplus \delta,
\end{equation*}
where $\mu$ and $\delta$ are topological $\HB$-vector bundles, $\rank
\mu = 1$ and $\delta$ is trivial. Since
\begin{equation*}
\xi(8) \cong \mu(2) \oplus \delta(2),
\end{equation*}
it suffices to prove that the $\HB$-vector bundle $\mu(2)$ admits a
stratified-algebraic structure. Note that
\begin{equation*}
c_4( (\mu(2))_{\CB} ) = c_4 ( (\xi(8))_{\CB} ) = 8c_4(\xi_{\CB}) + 28
c_2(\xi_{\CB}) \cupproduct c_2(\xi_{\CB}) = 0
\end{equation*}
in $H^8(X;\Z)$. Now, interpreting $c_4(\mu(2)) = e( (\mu(2))_{\R} )$ as
an obstruction, we get a nowhere vanishing continuous section of
$\mu(2)$. In other words, the $\HB$-line bundle $\mu$ is generated by two
continuous sections. It follows that we can find a continuous map $g
\colon X \to \SB^4$ with $\mu \cong g^*\theta$. According to
Theorem~\ref{th-2-13}, the $\HB$-vector bundle $\mu \oplus \mu = \mu(2)$
admits a stratified-algebraic structure. The proof of Case~\ref{case-1}
is complete.

\begin{case}\label{case-2}
Suppose that $\F=\CB$.
\end{case}

Since $n \geq 3$, we have $a(n, \CB) = a(n, \HB)+1$. Hence it suffices
to apply Case~\ref{case-1} and Lemma~\ref{lem-2-3} to the $\HB$-vector
bundle $\HB \otimes \xi$.

\begin{case}\label{case-3}
Suppose that $\F=\R$.
\end{case}

Since $n \geq 2$, we have $a(n, \R) = a(n, \CB)+1$. Hence it suffices to
apply Case~\ref{case-2} and Lemma~\ref{lem-2-4} to the $\CB$-vector
bundle $\CB \otimes \xi$.

The proof is complete.
\end{proof}

\section{Line bundles}\label{sec-3}

In this short section we concentrate our attention on $\CB$-line
bundles. For any real algebraic variety $X$, let $\VBC^1(X)$ denote the
group of isomorphism classes of topological $\CB$-line bundles on $X$
(with operation induced by tensor product). Let $\VBCstr^1(X)$ be the
subgroup of $\VBC^1(X)$ consisting of the isomorphism classes of all
$\CB$-line bundles admitting a stratified-algebraic structure. Since $X$
has the homotopy type of a compact polyhedron \cite[pp.~217, 225]{bib7},
the group $\VBC^1(X)$ is finitely generated, being isomorphic to the
cohomology group $H^2(X;\Z)$. In particular, the quotient group
\begin{equation*}
\GammaC^1(X) \coloneqq \VBC^1(X) / \VBCstr^1(X)
\end{equation*}
is finitely generated. Thus the group $\GammaC^1(X)$ is finite if and
only if
\begin{equation*}
r\GammaC^1(X) = 0
\end{equation*}
for some positive integer $r$. Furthermore, the latter condition holds if and
only if for every topological $\CB$-line bundle $\lambda$ on $X$ its
$r$th tensor power $\lambda^{\otimes r}$ admits a stratified-algebraic
structure.

\begin{proposition}\label{prop-3-1}
Let $X$ be a real algebraic variety. For any topological $\CB$-line
bundle $\lambda$ on $X$ and positive integer $r$, if $\lambda(r)$ admits
a stratified-algebraic structure, then so does $\lambda^{\otimes r}$.
\end{proposition}

\begin{proof}
If the $\CB$-vector bundle $\lambda(r)$ admits a stratified-algebraic
structure, then so does the $\CB$-line bundle $\det \lambda(r)$, cf.
\cite[Proposition~3.15]{bib30}. Here $\det \lambda(r)$ stands for the
$r$th exterior power of $\lambda(r)$. The proof is complete since the
$\CB$-line bundles $\det \lambda(r)$ and $\lambda^{\otimes r}$ are
isomorphic.
\end{proof}

As a consequence, we obtain the following.

\begin{corollary}\label{cor-3-2}
Let $X$ be a compact real algebraic variety. If $r$ is a positive
integer and $r\GammaC(X)=0$, then $r\GammaC^1(X)=0$.
\end{corollary}

\begin{proof}
It suffices to make use of Propositions~\ref{prop-1-2}
and~\ref{prop-3-1}.
\end{proof}

\begin{corollary}\label{cor-3-3}
For any compact real algebraic variety $X$ of dimension at most $8$, the
group $\GammaC^1(X)$ is finite and
\begin{equation*}
2^{a(\dim X, \CB) + a(X)} \GammaC^1(X) = 0,
\end{equation*}
where $a(X)=0$ if $\dim X \leq 7$ and $a(X) = 2$ if $\dim X = 8$.
\end{corollary}

\begin{proof}
This follows from Theorem~\ref{th-1-6} and Corollary~\ref{cor-3-2}.
\end{proof}

A different proof of Corollary~\ref{cor-3-3} for varieties of dimension
at most $5$ is given in \cite{bib29}. It is plausible that
$2\GammaC^1(X)=0$ for every compact real algebraic variety $X$, cf.
\cite[Conjecture~B, Proposition~1.5]{bib29}. This is confirmed by
Corollary~\ref{cor-3-3} for $\dim X \leq 4$. Without restrictions on the
dimension of $X$ we have the following.

\begin{theorem}\label{th-3-4}
Let $X$ be a compact real algebraic variety with $\Hsph^2(X;\Z) =
H^2(X;\Z)$. Then the group $\GammaC^1(X)$ is finite and
$2\GammaC^1(X)=0$.
\end{theorem}

\begin{proof}
According to Lemma~\ref{lem-2-12},
\begin{equation*}
2H^2(X;\Z) \subseteq \HCstr^2(X;\Z).
\end{equation*}
Hence for any topological $\CB$-line bundle $\lambda$ on $X$, the Chern
class $c_1(\lambda^{\otimes 2}) = 2c_1(\lambda)$ is in $\HCstr^2(X;\Z)$.
In view of \cite[Proposition~8.6]{bib30}, the $\CB$-line bundle
$\lambda^{\otimes 2}$ admits a stratified-algebraic structure. Thus $2
\GammaC^1(X) = 0$, as asserted.
\end{proof}

The following special case is of interest.

\begin{corollary}\label{cor-3-5}
Let $X$ be a compact real algebraic variety. If each connected component
of $X$ is homotopically equivalent to $\SB^{d_1} \times \cdots \times
\SB^{d_n}$ for some positive integers $d_1, \ldots, d_n$, then the group
$\GammaC^1(X)$ is finite and $2\GammaC^1(X)=0$.
\end{corollary}

\begin{proof}
Since $\Hsph^2(X;\Z) = H^2(X;\Z)$, it suffices to apply
Theorem~\ref{th-3-4}.
\end{proof}

According to \cite[Example~7.10]{bib30}, there exists a nonsingular real
algebraic variety $X$ diffeomorphic to the $n$-fold product $\SB^1
\times \cdots \times \SB^1$, $n \geq 3$, with $\GammaC^1(X) \neq 0$.

\phantomsection
\addcontentsline{toc}{section}{\refname}
\nocite{*}
\printbibliography

\end{document}